\documentclass{amsart}
\usepackage[american]{babel}
\usepackage{amssymb,amsmath,mathrsfs}
\usepackage{nicefrac}
\usepackage{hyperref}
\usepackage[normalem]{ulem}
\hypersetup{
  colorlinks  = true, 
  urlcolor    = black, 
  linkcolor   = black, 
  citecolor   = black 
}

\newcommand{\FF}{\mathbb{F}}
\newcommand{\PP}{\mathbb{P}}

\newcommand{\C}{\mathcal{C}}
\newcommand{\TT}{\mathbb{T}}
\newcommand{\bx}{\mathbf{x}}
\newcommand{\bz}{\mathbf{z}}
\newcommand{\bw}{\mathbf{w}}
\newcommand{\reg}{\operatorname{reg}}

\newcommand{\ev}{\operatorname{ev}}
\newcommand{\set}[1]{\left \{ #1 \right \}}

\newcommand{\ts}{\textstyle}

\newtheorem{theorem}{Theorem}[section]
\newtheorem{lemma}[theorem]{Lemma}

\newtheorem{prop}[theorem]{Proposition}
\newtheorem*{claim*}{Claim}

\theoremstyle{definition}
\newtheorem{definition}[theorem]{Definition}

\newtheorem{example}[theorem]{Example}

\begin{document} 

\title[The minimum distance of a parameterized code over an even cycle]%
{The minimum distance of a parameterized\\ code over an even cycle}

\author{Eduardo Camps-Moreno}
\address[Eduardo Camps-Moreno]{Department of Mathematics, Virginia Tech, Blacksburg, VA
USA}
\email{e.camps@vt.edu}

\author{Jorge Neves}
\address[Jorge Neves]{University of Coimbra, Department of Mathematics, CMUC, 3000-143 Coimbra, Portugal}
\email{neves@mat.uc.pt}

\author{Eliseo Sarmiento}
\address[Eliseo Sarmiento]{Escuela Superior de F\'isica y Matem\'aticas \\ Instituto Polit\'ecnico Nacional\\ Mexico City, Mexico}
\email{esarmiento@ipn.mx}

\thanks{Partially supported by the Centre for Mathematics of the University of Coimbra (funded by the Portuguese Government through FCT/MCTES, DOI 10.54499/UIDB/00324/2020). The authors use Macaulay2 \cite{M2} in the computations of examples.}

\thanks{\emph{Data availability}. Data sharing not applicable to this article as no datasets were
generated or analyzed during the current study.}

\makeatletter
\@namedef{subjclassname@2020}{%
  \textup{2020} Mathematics Subject Classification}
\makeatother


\begin{abstract}
We compute the minimum distance of the parameterized code of order $1$ over an even cycle.
\end{abstract}
\maketitle

\section{Introduction}
Defined by Renter\'ia, Simis and Villarreal in \cite{ReSiVi11}, and by Gonz\'alez and Renter\'ia in \cite{GoRe07}, 
para\-me\-te\-rized linear codes are evaluation codes over a projective toric subset
parameterized by a set of Laurent monomials, over a finite field $\FF$.
From the start, the case when the exponent vectors of the Laurent monomials 
form the incidence matrix of a graph is a case of major interest. Let us explain this in detail. 
Let $G$ be a simple graph, with vertex set $V_G=\set{1,2,\dots,n}$ and $E_G=\set{e_1,\dots,e_s}$. 
Fix a finite field, $\FF$, of order $q$ and,
for any $\bx=(x_1,\dots,x_n) \in (\FF^*)^n$ and any edge $e_k=\set{i,j}\in E_G$,
denote $e_k(\bx)=x_ix_j$. Denote the subset of $\PP^{s-1}$ consisting 
of the points with nonzero coordinates by $\TT^{s-1}$.

\begin{definition}\label{def: the set of points}
The \emph{projective toric subset 
parameterized by $G$} is
$$
X = \set{(e_1(\bx):\dots : e_s(\bx))\in \PP^{s-1} : \bx \in (\FF^*)^n}\subseteq \TT^{s-1}.
$$
\end{definition}
To avoid a the trivial case, we assume throughout this article that 
$|\FF|=q>2$.
Denoting by $S=\FF[t_1,\dots,t_s]$ the polynomial ring with variables in bijection with 
$E_G$ and coefficients in $\FF$ and choosing an ordering of the points of 
$X = \set{P_1,\dots,P_m}$, define, for each $d\geq 0$, a linear 
map $\ev_d\colon S_d \to \FF^m$ given by
$$
\ts f\mapsto \left (\frac{f}{t_1^d}(P_1),\dots,\frac{f}{t_1^d}(P_m)\right ). 
$$
\begin{definition}\label{def: the codes}
The image of $\ev_d$ is called the parameterized code of order $d$ over $G$ and $\FF$. We denote 
these linear subspaces by $C_X(d)\subseteq \FF^m$. 
\end{definition}

Up to linear equivalence, the codes $C_X(d)$ do not depend on the choice 
of ordering of the edges of graph nor on the ordering of the points of $X$. 
\medskip

Denote the vanishing ideal of $X$ by $I(X)\subseteq S=\FF[t_1,\dots,t_s]$.
It is by now well-known that the properties of the family of evaluation codes
given by $C_X(d)$, as $d\geq 0$, tie into the algebraic properties of the quotient $S/I(X)$. 
The three basic parameters of a code $C\subseteq \FF^m$ are the length ($m$, the 
dimension of their ambient space), the dimension and the minimum distance, \emph{i.e.}, 
$$
\min\set{\|\bw\| : 0\not =\bw \in C }.
$$
The length of $C_X(d)$, $m=|X|$, which is constant as $d\geq 0$, can be given by the multiplicity degree of $S/I(X)$.
An explicit formula for this parameter is known. If $b_0(G)$ denotes the number of connected components of $G$
and $\gamma$ is the number of those that are non-bipartite, then 
\begin{equation}\label{eq: length}
\renewcommand{\arraystretch}{1.3}
m = |X| = \deg S/I(X) = 
\left \{
\begin{array}{ll}
\frac{1}{2^{\gamma-1}}(q-1)^{n-m+\gamma-1} & \text{if $\gamma\geq 1$ and $q$ is odd,} \\ 
(q-1)^{n-m+\gamma-1} & \text{if $\gamma\geq 1$ and $q$ is even,} \\
(q-1)^{n-m-1}, & \text{if $\gamma=0$}\\
\end{array}
\right. 
\end{equation}
(\emph{cf}.~\cite[Corollary~3.8]{ReSiVi11} and \cite[Theorem 3.2]{NeVPVi15}). 
Concerning the parameter dimension, from Definition~\ref{def: the codes}, it follows that 
$$
\dim_\FF C_X(d) = \dim_\FF \left ( S/I(X) \right )_d.
$$
In other words the dimension function of $C_X(d)$  
coincides with the Hilbert function of $S/I(X)$. Unlike the case of the length, 
no formula holding for a general $G$ is known for the dimension function. 
We know of formulas for the dimension function 
only in the following cases: \emph{i}) when $X$ coincides with the torus $\TT^{s-1}$, \emph{i.e.}, 
when $G$ is either a forest or a non-bipartite unicyclic graph
\cite[Proposition~2.2]{DuReTR01}; \emph{ii}) when $G$ is a complete bipartite graph 
\cite[Theorem~5.2]{GR08}; \emph{iii}) when $G$ is an even cycle and $\FF$ is a field 
with $q=3$ elements --- \emph{cf}.~\cite[Theorem~2.8]{NeVP23}.
On the other hand, the behaviour of the Hilbert function of $S/I(X)$ also yields another 
important piece of information about the family of linear codes $C_X(d)$. 
Since $S/I(X)$ is a Cohen--Macaulay, dimension $1$ graded ring, we know that 
the Hilbert function eventually becomes constant and equal to $m$. This means 
that for a certain value of $d$ onward the codes $C_d(X)$ are trivial, as
they coincide with their ambient space. This value of $d$ is the regularity index 
of $S/I(X)$ or, in this case equivalently, the Castelnuovo--Mumford regularity 
of $S/I(X)$. For the past years, several authors have focused on the computation of the 
regularity of $S/I(X)$ ---
\emph{cf.}~\cite{GR08, GoReSa13, NeVP14,NeVPVi15,Ne20,MaNeVPVi20}. 
When $X$ is equal to $\TT^{s-1}$, $I(X)$ is equal to the complete intersection 
$$
(t_1^{q-1}-t_s^{q-1},\dots,t_{s-1}^{q-1}-t_s^{q-1})
$$
and, accordingly, $\reg S/I(X) = (s-1)(q-2)$ --- \emph{cf}.~\cite{GRH03}.
If $G$ is bipartite, $2$-connected and affords a nested ear decomposition, then the
regularity is equal to 
$$
\ts \frac{n+\epsilon-3}{2}(q-2),
$$
where $\epsilon$ denotes the number of even ears --- \emph{cf}.~\cite[Theorem~4.4]{Ne20}. 
If $q= 3$, the regularity is equal to the maximum cardinality of a subset of 
edges that contains no more than half the edges of any Eulerian, 
even cardinality, subgraph of $G$ --- \emph{cf}.~\cite[Theorem~4.13]{Ne23}.
The regularity is known for a few more special families of graphs such as 
complete graphs and complete bipartite graphs ---
\emph{cf}.~\cite{GoReSa13} and \cite{GR08}, respectively --- however, so far, 
we know of no formula for the regularity of $S/I(X)$ that holds 
for a general graph and $q>3$.

The minimum distance of $C_X(d)$ has remained unknown, for most cases. 
When $X$ coincides with the torus and $d<\reg S/I(X) =(q-1)(s-1)$, we know that the minimum distance of $C_X(d)$ is equal to 
$$
(q-1)^{s-(k+2)}(q-1-\ell)
$$
where $k$ and $\ell$ are such that $d=k(q-2)+\ell$, $k\geq 0$ and $1\leq \ell \leq q-2$ --- \emph{cf}.~\cite[Theorem~3.5]{SVPV11}.
This formula can be used to express the minimum distance for the case when $G$ 
is a complete bipartite graph, \emph{cf}.~\cite{GR08}. For no other classes of graphs do we know a formula for the minimum distance. 
A complete answer to this question seems, at present, a very difficult problem. 

The purpose  of this article is to give a formula for the minimum distance of $C_X(1)$ when $G$ is an even cycle. 
This result is significant not only because it provides a closed formula for an unknown parameter, 
but also because it lays the groundwork for exploring minimum distances across a wider 
range of graph structures. In the next section we describe the corresponding set $X$ and linear code $C$, 
together with the length and dimension of the latter. In Section~\ref{sec: the result}, we state and prove 
our main result, Theorem~\ref{thm: the result}.

\section{Preliminaries}

\subsection{The set of points} Let $G=\C_{2k}$ be an even cycle of length $2k\geq 4$. 
The toric subset parameterized by 
$\C_{2k}$ is given by: 
\begin{equation}\label{eq: the set}
X = \set{(x_1x_2:x_2x_3: \dots:x_{2k-1}x_{2k}:x_{2k}x_1) \in \PP^{2k-1} : x_i \in \FF^*, 1\leq i\leq 2k}.
\end{equation}
Using \eqref{eq: length} we deduce that $X$ has cardinality equal to $m=(q-1)^{2k-2}$.

\subsection{The code}
Recalling that $S=\FF[t_1,\dots,t_{2k}]$, the code $C_X(1)$ 
is the image of the map $\ev_1\colon S_1 \to \FF^m$ given by 
$$
\ts t_i\mapsto \left (\frac{t_i}{t_1}(P_1),\dots,\frac{t_i}{t_1}(P_m)\right ).
$$
for $i=1,\dots,2k$. For the remainder of this article we will denote this code by $C$. 
Since by \cite[Theorem 5.9]{NeVPVi15}, the vanishing ideal of $X$ is generated
by binomials of degree $\geq 2$, we deduce that 
$$
\dim_\FF C = \dim_\FF (S/I(X))_1 = \dim_\FF S_1 = 2k.
$$
Hence $C$ is a code of length 
$m=(q-1)^{2k-2}$ and dimension $2k$.

\subsection{Number of zeros of forms}
The set $X$ is the image of a map 
$\TT^{2k-1} \to \TT^{2k-1}$ that sends $(x_1:\cdots :x_{2k})$ to 
$(x_1x_2:\cdots : x_{2k-1}x_{2k}:x_{2k}x_1)$. This is a group homomorphism 
which, according to \cite[Lemma 3.1]{NeVPVi15},
has kernel of cardinality $q-1$. Hence, given a homogeneous polynomial $F\in \FF[t_1,\dots,t_{2k}]$, 
the computation of the cardinality of $Z(F)\cap X$ can be achieved by counting 
the number of zeros of 
$$
f = F(x_1x_2,\dots,x_{2k-1}x_{2k}, x_{2k}x_1) \in \FF[x_1,\dots,x_n]
$$
on $\TT^{2k-1}$ and dividing this number by $(q-1)$. Similarly, the cardinality 
of $Z(f)\cap \TT^{2k-1}$ is equal to the number of $(a_1,\dots,a_{2k})\in (\FF^*)^{2k}$
such that $f(a_1,\dots,a_{2k})=0$, divided by $q-1$.

\begin{definition}
Given $f_1,\dots,f_\ell\in \FF[x_1,\dots,x_n]$, denote by $\bz(f_1,\dots,f_\ell)$ the number of 
$(a_1,\dots,a_{2k})\in (\FF^*)^{2k}$ such that 
$$
f_1(a_1,\dots,a_{2k}) = \cdots = f_\ell(a_1,\dots,a_{2k}) = 0. 
$$
\end{definition}

By the previous discussion we get following lemma. 

\begin{lemma}\label{lemma: zeros affine}
Given $F\in S$, if  $f=F(x_1x_2,\dots,x_{2k-1}x_{2k},x_{2k}x_1)\in \FF[x_1,\dots,x_n]$,
the cardinality of $Z(F)\cap X$ is $\bz(f)/(q-1)^2$.
\end{lemma}

\section{The result}\label{sec: the result}

\begin{theorem}\label{thm: the result}
Let $X\subseteq \PP^{2k-1}$ be the toric subset parameterized by a cycle of length $2k\geq 4$, 
let $m=|X|=(q-1)^{2k-2}$, and let $C=C_X(1)\subseteq \FF^{m}$ be the corresponding evaluation code of order $1$.  
Then the minimum distance of $C$ is
$$
\renewcommand{\arraystretch}{1.75}
\left \{
\begin{array}{ll}
(q-1)^{2k-3}(q-2), & \text{if $k=4$ and $q> 3$, or $k\geq 5$,}\\
\frac{(q-1)^{2k}-q^k(q-2)-1}{q(q-1)}, & \text{if $k=2$, or $k=3$, or $k=4$ and $q=3$}.
\end{array}
\right.
$$
\end{theorem}

\medskip

Since $C$ is the subspace of $\FF^m$ obtained by evaluating forms of degree one at the points of 
$X$, to prove Theorem~\ref{thm: the result}, we need to show that
$$
\max_{F\in S_1} |Z(F)\cap X| = 
\renewcommand{\arraystretch}{1.75}
\left \{
\begin{array}{ll}
(q-1)^{2k-3}, & \text{if $k=4$ and $q> 3$, or $k\geq 5$,}\\
\frac{(q-1)^{2k-1}+q^k(q-2)+1}{q(q-1)}, & \text{if $k=2$, or $k=3$, or $k=4$ and $q=3$}.
\end{array}
\right.
$$
To estimate $|Z(F)\cap X|$, with $F\in S_1$, we will distinguish between 
two cases, corresponding to Propositions~\ref{prop: zeros of incomplete linear form} and 
\ref{prop: zeros of a complete linear form}, below.

\begin{prop}\label{prop: zeros of incomplete linear form}
If $F= \alpha_1t_1+\cdots+\alpha_{2k}t_{2k} \in \FF[t_1,\dots,t_s]$ is a nonzero linear form  
such that $\alpha_i=0$, for some $i$, then 
$$
|Z(F)\cap X|\leq  (q-1)^{2k-3}
$$
and equality is attained, for example, for $F=t_1-t_2$.
\end{prop}

\begin{proof}
Let $f=F(x_1x_2,\dots,x_{2k-1}x_{2k},x_{2k}x_1)$. By our assumptions there exists $i$ such that 
$f=h+\beta x_ix_{j}$ with $\beta\not = 0$ and $h\in \FF[x_1,\dots,\widehat{x_i},\dots,x_{2k}]$.
This implies that if an element of $(\FF^*)^{2k}$ is a zero of $f$ then $h$ does not vanish and 
the $i$th coordinate is determined by the value of the remaining coordinates. In other words,  
$$
\bz(f) \leq (q-1)^{2k-1},
$$
which, by Lemma~\ref{lemma: zeros affine}, says that $|Z(F)\cap X| \leq (q-1)^{2k-3}$.
When $F=t_1-t_2$ we get $f=x_1x_2-x_2x_3$ and, since $x_1x_2$ does not vanish on any element of 
$(\FF^*)^{2k}$,
$$
\bz(f) = (q-1)^{2k-1} \implies |Z(F)\cap X| = (q-1)^{2k-3}. \qedhere
$$
\end{proof}

To deal with the case when all $\alpha_i$ are nonzero, we will need the next lemma.

\begin{lemma}\label{lemma: zeros of path}
Let $2\leq r\leq 2k$. If $g=\sum_{i=1}^{r-1} \beta_i x_ix_{i+1}\in \FF[x_1,\dots,x_{2k}]$ is such that  
$\beta_1,\dots,\beta_{r-1}$ are nonzero, then 
\begin{equation}\label{eq: zeros of path}
\bz(g) =  q^{-1}[(q-1)^r+(-1)^{r-1}(q-1)^2](q-1)^{2k-r}.
\end{equation}
\end{lemma}

\begin{proof}
Let us use induction on $r$. If $r=2$ then $g=\beta_1x_1x_2$, with $\beta_1\not = 0$, and  
hence \mbox{$\bz_2(g)=0$}, which agrees with the formula given. Fix $r<2k$ and assume that the statement  
holds for $r$. Given 
$$
\ts g = \sum_{i=1}^r \beta_i x_ix_{i+1}\in \FF[x_1,\dots,x_{2k}]
$$ 
with $\beta_1,\dots,\beta_r\not = 0$, let us write $g = h + \beta_{r}x_rx_{r+1}$, where 
$$
\ts h= \sum_{i=1}^{r-1} \beta_i x_ix_{i+1}.
$$
Now, an element of $(\FF^*)^{2k}$ is a zero of 
$g$ if and only if it not a zero of $h$ and in this case its $r+1$ coordinate is determined by $0=h+\beta_rx_rx_{r+1}$. 
Hence
$$
\bz(g) = (q-1)^{2k-1} - \bz(h)(q-1)^{-1},
$$
which, by induction is equal to 
$$
\renewcommand{\arraystretch}{1.3}
\begin{array}{l}
(q-1)^{2k-1} - q^{-1}[(q-1)^r+(-1)^{r-1}(q-1)^2](q-1)^{2k-r-1} \\ 
= q^{-1} [q(q-1)^{r}  - (q-1)^{r}-(-1)^{r-1}(q-1)^2](q-1)^{2k-r-1} \\
= q^{-1} [(q-1)^{r+1}+(-1)^{r}(q-1)^2](q-1)^{2k-r-1}.\qedhere
\end{array}
$$
\end{proof}

\begin{prop}\label{prop: zeros of a complete linear form}
If $F= \alpha_1t_1+\cdots+\alpha_{2k}t_{2k} \in \FF[t_1,\dots,t_s]$ is   
such that all $\alpha_i$ are nonzero, then 
$$
|Z(F)\cap X| \leq \frac{(q-1)^{2k-1}+q^k(q-2)+1}{q(q-1)},
$$
with equality if and only if $\alpha_1\alpha_3\cdots \alpha_{2k-1} = (-1)^k \alpha_2\alpha_4\cdots \alpha_{2k}$.
\end{prop}

\begin{proof}
Let $f = F(x_1x_2,\dots,x_{2k-1}x_{2k},x_{2k}x_1)$. By Lemma~\ref{lemma: zeros affine}, we need 
to show
\begin{equation}\label{eq: number of zeros}
\bz(f)\leq q^{-1}\bigl[(q-1)^{2k}+q^k(q-2)(q-1) + q-1\bigr]
\end{equation}
with equality if and only if 
$\alpha_1\alpha_3\cdots \alpha_{2k-1} = (-1)^k \alpha_2\alpha_4\cdots \alpha_{2k}$. We will argue 
by induction on $k\geq 2$. Assume $k=2$, so that 
$$
\begin{array}{l}
f = \alpha_1x_1x_2 + \alpha_2x_2x_3+\alpha_3x_3x_4 + \alpha_4x_4x_1\\
\phantom{f} = \alpha_1x_1x_2 + \alpha_2 x_2x_3 + (\alpha_3 x_3+\alpha_4 x_1)x_4\\
\phantom{f} = g + hx_4
\end{array}
$$
where we are denoting $g=\alpha_1x_1x_2 + \alpha_2 x_2x_3$ and $h=\alpha_3 x_3+\alpha_4 x_1$.
If we let $x_1,x_2,x_3$ vary in $\FF^*$, a zero of $f$ has $x_4\in \FF^*$ determined by the values 
of $g$ and $h$, unless both quantities are zero. Thus, using Lemma~\ref{lemma: zeros of path},
$$
\renewcommand{\arraystretch}{1.3}
\begin{array}{l}
\bz(f) = \bz(g,h) + (q-1)^3 - \bz(g)(q-1)^{-1}-\bz(h)(q-1)^{-1}+\bz(g,h)(q-1)^{-1} \\
\phantom{\bz(f)} = (q-1)^{-1} q\bz(g,h)+ (q-1)^3 - \bz(g)(q-1)^{-1}-\bz(h)(q-1)^{-1}\\
\phantom{\bz(f)} = (q-1)^{-1} q\bz(g,h)+(q-1)^3- q^{-1}[(q-1)^3+(q-1)^2] -(q-1)^2\\
\phantom{\bz(f)} = (q-1)^{-1} q\bz(g,h)+(q-1)^3- 2(q-1)^2\\
\phantom{\bz(f)} = (q-1)^{-1}q\bz((\alpha_1-\alpha_2\alpha_4\alpha_3^{-1})x_1x_2, \alpha_3x_3+\alpha_4x_4) +(q-1)^3- 2(q-1)^2\\
\phantom{\bz(f)} \leq q(q-1)^2 +(q-1)^3-2(q-1)^2\\
\phantom{\bz(f)} = (q-1)^2[2q-3]
\end{array}
$$
which coincides with \eqref{eq: number of zeros} for $k=2$. Moreover, equality above holds if and only if 
$$
\alpha_1-\alpha_2\alpha_4\alpha_3^{-1} = 0 \iff \alpha_1\alpha_3 = (-1)^2 \alpha_2\alpha_4.
$$
Consider the induction step. Let $f\in \FF[x_1,\dots,x_{2k+2}]$ be 
$$
\begin{array}{l}
f = \alpha_1x_1x_2 +\cdots + \alpha_{2k}x_{2k}x_{2k+1}+ \alpha_{2k+1} x_{2k+1}x_{2k+2} + \alpha_{2k+2}x_{2k+2}x_1 \\
\phantom{f} = g + hx_{2k+2},
\end{array}
$$
where we $g=\alpha_1x_1x_2+\cdots +\alpha_{2k}x_{2k}x_{2k+1}$ and $h=\alpha_{2k+1}x_{2k+1}+\alpha_{2k+2}x_1$.
Arguing as in the case of $k=2$,
\begin{equation}\label{eq: L446}
\bz (f) = (q-1)^{-1}q\bz(g,h)+(q-1)^{2k+1}-\bz(g)(q-1)^{-1}-\bz(h)(q-1)^{-1}. 
\end{equation}
Since  
$\bz(g,h) = \bz(f',h)$, where 
$$
f' = \alpha_1x_1x_2 +\cdots + \alpha_{2k-1}x_{2k-2}x_{2k-1}-\alpha_{2k}\alpha_{2k+2}\alpha_{2k+1}^{-1}x_{2k}x_1,
$$
by induction, 
$$
\bz(g,h) \leq q^{-1}[(q-1)^{2k}+q^k(q-2)(q-1) + (q-1)](q-1)
$$
with equality if and only if 
\begin{equation}\label{eq: condition}
\renewcommand{\arraystretch}{1.3}
\begin{array}{c}
\alpha_1\alpha_3 \cdots \alpha_{2k-1} = (-1)^k \alpha_2\alpha_4 \cdots \alpha_{2k-2}[-\alpha_{2k}\alpha_{2k+2}\alpha_{2k+1}^{-1}]\\
\iff \alpha_1\alpha_3 \cdots \alpha_{2k-1} \alpha_{2k+1} = (-1)^{k+1} \alpha_2\alpha_4 \cdots \alpha_{2k-2}\alpha_{2k}\alpha_{2k+2}.
\end{array}
\end{equation}
Using this, we deduce that
$$
\renewcommand{\arraystretch}{1.3}
\begin{array}{l}
(q-1)^{-1}\bigl[q\bz(g,h)-\bz(h)\bigr]\leq (q-1)^{2k}+q^k(q-2)(q-1)+(q-1)-(q-1)^{2k}\\
\phantom{(q-1)^{-1}\bigl[q\bz(g,h)-\bz(h)\bigr]} = q^k(q-2)(q-1)+(q-1)
\end{array}
$$
Together with Lemma~\ref{lemma: zeros of path}, this yields from \eqref{eq: L446}
$$
\renewcommand{\arraystretch}{1.3}
\begin{array}{l}
\bz(f)\leq q^k(q-2)(q-1)+(q-1) + (q-1)^{2k+1}
-q^{-1}[(q-1)^{2k+1}+(q-1)^2]\\
\phantom{\bz(f)} = q^{-1}[(q-1)^{2k+2}+q^{k+1}(q-2)(q-1)+(q-1)],
\end{array}
$$
with equality if and only if \eqref{eq: condition} holds.
\end{proof}


We are now ready to prove the main result. 

\begin{proof}[Proof of Theorem~\ref{thm: the result}]
By Propositions~\ref{prop: zeros of incomplete linear form} and \ref{prop: zeros of a complete linear form}
$$
\ts \max\limits_{F\in S_1}|Z(F)\cap X| = \max\left\{(q-1)^{2k-3},\frac{(q-1)^{2k-1}+q^k(q-2)+1}{q(q-1)}\right\}.
$$
Observe that  
$$
\ts (q - 1)^{2k - 3} - \frac{(q - 1)^{2k - 1} + q^k(q - 2) + 1}{q(q - 1)} 
= \frac{(q - 1)^{2k - 2} - q^k(q - 2) - 1}{q(q - 1)}
$$
and denote the expression $(q - 1)^{2k - 2} - q^k(q - 2)$ in the 
numerator of this fraction by $\Delta(k,q)$. 
We want to show that
$$
\Delta(k,q) >0 \iff \bigl(k>5 \text{ and } q>2\bigr) \text{ or } \bigl(k=4 \text{ and } q>3 \bigr).
$$
Since 
$$
\renewcommand{\arraystretch}{1.3}
\begin{array}{l}
\Delta(2,q) = (q-1)^2-q^2(q-2)<q^2-q^2(q-2) < 0, \\
\Delta(3,q) = (q-1)^4-q^3(q-2)<q[(q-1)^3-q^2(q-2)] = q[-q^2+3q-1] < 0,
\end{array}
$$
we may reduce to $k\geq 4$. By direct computation, $\Delta(4,3) = 2^6-3^4<0$. Hence, if $k=4$ we may reduce to $q>3$. 
In this situation, $q-1\geq \frac{3}{4}q$ and this yields:
$$
\ts \Delta(4,q) = (q-1)^6-q^4(q-2) \geq \frac{3^5}{4^5}q\cdot  q^4(q-2) -q^4(q-2),
$$
the right side of which is positive if $3^5q>4^5$ which is equivalent to $q\geq 5$. Since, by direct computation, 
$\Delta(4,4) = 3^6 - 4^4\cdot 2>0$ we conclude that $\Delta(4,q)>0$ if and only if $q>3$.
Finally, assume that $k\geq 5$. Since 
$$
\renewcommand{\arraystretch}{1.4}
\begin{array}{l}
\Delta(k,q) = (q-1)^{2k-2}-q^k(q-2) \\
\phantom{\Delta(k,q)} = (q-1)^2\Delta(k-1,q) + (q-1)^2 q^{k-1}(q-2) - q^k(q-2) \\
\phantom{\Delta(k,q)} = (q-1)^2\Delta(k-1,q) + [q^2 -3q +1]q^{k-1}(q-2) \\
\end{array}
$$
and $q^2 -3q +1>0$ for $q>2$, from the positivity of $\Delta(4,q)$, for $q>3$ we get
$\Delta(k,q)>0$, for $q>3$. Finally, if $q=3$
$$
\textstyle\Delta(k,3) = 2^{2k-2}-3^k >0 \iff k>\frac{\ln(4)}{\ln(4)-\ln(3)} \iff k\geq 5. \qedhere
$$
\end{proof}

\begin{example}
Consider $G=\C_{4}$, where $k=2$. Then $C$ is a code of length $(q-1)^2$ and dimension $4$.
According to Theorem \ref{thm: the result}, its minimum distance is 
$$
\textstyle \frac{(q-1)^{4}-q^2(q-2)-1}{q(q-1)} = (q-2)^2.
$$
This agrees with the results of \cite{GR08} and \cite{SVPV11}. Indeed, since $\C_4$ is  
a complete bipartite graph $\mathcal{K}_{2,2}$, by \cite[Theorem~5.5]{GR08}, the minimum distance 
of $C$ is the square of the minimum distance of the evaluation code over $\TT^1$ of order $1$, 
which, by \cite[Theorem~3.5]{SVPV11}, is equal to $(q-2)$.
\end{example}

\subsection*{Acknowledgments}
The authors thank Manuel Gonz\'alez-Sarabia for his helpful comments.

\end{document}